\documentclass[12pt]{amsart}
\usepackage[T1]{fontenc}
\usepackage[utf8]{inputenc}
\usepackage[margin=1in]{geometry}
\usepackage{amsmath, amsfonts, amssymb, amsthm}
\usepackage{etoolbox} 
\usepackage[shortlabels]{enumitem}
\usepackage{tikz}
\usetikzlibrary{decorations.pathreplacing,calc}
\usepackage{verbatim}
\usepackage{thmtools}
\usepackage[dvipsnames]{xcolor}
\usepackage{float}
\usepackage{natbib}
\usepackage{parskip}
\usepackage[normalem]{ulem}
\usepackage[pdfencoding=auto, psdextra]{hyperref}
\usepackage{xspace}

\usepackage{pgfplots}
\pgfplotsset{compat=1.18}
\usepackage[nameinlink]{cleveref}

\newcommand{\drawdomboard}[6][0]{
\begin{scope}[xshift=#1cm]

\draw[step=1cm] (0,0) grid (#2,#3);

\foreach \x/\y in {#4} {
    \fill[gray!25]
        ($(\x,\y) + (-2,-1)$) rectangle ($(\x,\y)$);
    \draw[black, very thick]
        ($(\x,\y) + (-2,-1)$) rectangle ($(\x,\y)$);
}

\foreach \x/\y in {#5} {
    \fill[gray!25]
        ($(\x,\y) + (-1,-2)$) rectangle ($(\x,\y)$);
    \draw[black, very thick]
        ($(\x,\y) + (-1,-2)$) rectangle ($(\x,\y)$);
}

\foreach \x/\y in {#6} {
        \fill[blue!25]
            ($(\x,\y) + (-1,-2)$) rectangle ($(\x,\y)$);
        \draw[black, very thick]
            ($(\x,\y) + (-1,-2)$) rectangle ($(\x,\y)$);
    }

\end{scope}
}

            \newcommand{\boardTwoByThree}[1]{
            \begin{tikzpicture}[scale=0.55, baseline=(current bounding box.center)]
              \def\W{3}
              \def\H{2}
              \draw[thick] (0,0) rectangle (\W,\H);
              \draw (0,1) -- (\W,1);
              \foreach \x in {1,2} \draw (\x,0) -- (\x,\H);
              #1
            \end{tikzpicture}
            }
            
            \newcommand{\hdom}[2]{
              \fill[gray!25] (#2,#1) rectangle ({#2+2},{#1+1});
              \draw[very thick] (#2,#1) rectangle ({#2+2},{#1+1});
            }
            \newcommand{\vdom}[1]{
              \fill[blue!25] (#1,0) rectangle ({#1+1},2);
              \draw[very thick] (#1,0) rectangle ({#1+1},2);
            }

\newcommand{\boardTwoByFive}[1]{
\begin{tikzpicture}[scale=0.5, baseline=(current bounding box.center)]
  \def\W{5}
  \def\H{2}
  \draw[thick] (0,0) rectangle (\W,\H);
  \draw (0,1) -- (\W,1);
  \foreach \x in {1,2,3,4} \draw (\x,0) -- (\x,\H);
  #1
\end{tikzpicture}
}

\newcommand{\TOPFIXED}{\hdom{1}{0}\hdom{1}{2}}

\newtheorem{theorem}{Theorem}[section]
\newtheorem{lemma}[theorem]{Lemma}
\newtheorem{remark}[theorem]{Remark}
\newtheorem{corollary}[theorem]{Corollary}
\newtheorem{proposition}[theorem]{Proposition}
\newtheorem{definition}[theorem]{Definition}
\newtheorem{conjecture}[theorem]{Conjecture}
\newtheorem{problem}[theorem]{Problem}

\numberwithin{equation}{section}

\newcommand{\qfibonom}[2]{
\left[\!\begin{matrix} #1 \\ #2 \end{matrix}\!\right]_{\mathcal{F}}
}
\newcommand{\fibonom}[2]{
\left(\!\begin{matrix} #1 \\ #2 \end{matrix}\!\right)_{\mathcal{F}}
}
\newcommand{\Tm}{\mathcal{T}_{m,2}}
\def\Z{\mathbb{Z}}
\newcommand{\fibocat}[2]{\operatorname{FCat}_{#1,#2}}
\newcommand{\qfibocat}[2]{[\operatorname{FCat}_{#1,#2}]_q}
\newcommand{{\pow}}{r}

\def\bmref#1{{\normalfont\hyperref[#1]{\ref*{#1}}}\xspace}

\title{Unimodality of $q$-Fibonomial coefficients for small cases}

\author[B. B. Connelly]{Brendan B. Connelly}
\address{Department of Mathematics, UCLA, Los Angeles, California}
\email{\href{mailto:brendanconnelly@ucla.edu}{brendanconnelly@ucla.edu}}

\author[E. Ito]{Ezekiel Ito}
\email{\href{mailto:ezekielito@ucla.edu}{ezekielito@ucla.edu}}

\author[T. C. Martinez]{Thomas C. Martinez}
\email{\href{mailto:thomasmart@ucla.edu}{thomasmart@ucla.edu}}

\author[O. Shevchenko]{Olha Shevchenko}
\email{\href{mailto:shevchenko@ucla.edu}{shevchenko@ucla.edu}}

\author[K. Yang]{Kacey Yang}
\email{\href{mailto:kaceyyang@ucla.edu}{kaceyyang@ucla.edu}}

\date{\today}

\begin{document}

\makeatletter
\@namedef{subjclassname@2020}{
  \textup{2020} Mathematics Subject Classification}
\makeatother



\begin{abstract}
    Bergeron--Ceballos--K\"ustner introduced the $q$-Fibonomial coefficients \( \qfibonom{m+n}{n}\), gave a combinatorial interpretation of the $q$-Fibonomial coefficients via a weighted path-domino tiling model, and conjectured that these polynomials are unimodal. We prove the conjecture for $n\leq3$. For the $n=2$ case, we give a combinatorial proof of both unimodality and symmetry by defining a nearly symmetric saturated chain decomposition on the set of tilings. For all three cases, we give an algebraic proof. Finally, for the $n=3$ case, we establish a more general unimodality result for certain products of $q$-analogs and propose several related conjectures. 
\end{abstract}

\maketitle

\section{Introduction}

The $q$-Fibonomial coefficients, introduced by Bergeron--Ceballos--K\"ustner \cite{kustner}, are the Fibonacci analogs of the Gaussian polynomials. For $m,n \in \mathbb{N}$, define the $q$-Fibonomial as follows:
\begin{equation}\label{eq:qfibo}
\qfibonom{m+n}{n} := \frac{[F_{m+n}]^!_q}{[F_m]^!_q \cdot [F_n]^!_q},
\end{equation}
where $[F_n]^!_q :=\prod_{k=1}^n[F_k]_q$, and $[n]_q:=1+q+\dots+q^{n-1}$. \cite{kustner} shows $q$-Fibonomials are polynomials with non-negative integer coefficients by attaching $q$-weights to a variant of the path-domino tiling model of Sagan--Savage \cite{sagan2009combinatorial} described in \Cref{sec:background}.

The polynomials $\qfibonom{m+n}{n} = \sum_{k=0}^Nc_kq^k$ are symmetric, i.e., $c_k = c_{N-k}$, and Bergeron--Ceballos--K\"ustner further conjectured $\qfibonom{m+n}{n}$ are unimodal, i.e., $c_0\leq c_1\leq \dots \leq c_{M-1}\leq c_M\geq c_{M+1}\geq\dots\geq c_{N-1}\geq c_N$ for some $M$, mirroring the classical results for the Gaussian polynomials \cite{sylvester1878xxv,o1990unimodality}.

\begin{conjecture}[{\cite[Conjecture 2.5]{kustner}}]\label{conj:mbynunimodal}
    The polynomials $\qfibonom{m+n}{n}$ are unimodal.
\end{conjecture}
 
Our main result resolves the conjecture for $n \leq 3$:
 
\begin{theorem}\label{thm:mainthm}
    The polynomials $\qfibonom{m+1}{1}$, $\qfibonom{m+2}{2}$, and $\qfibonom{m+3}{3}$ are unimodal.
\end{theorem}

Unimodality is often proved using log-concavity \cite[p.~500]{stanley1989log}. However, $q$-Fibonomial coefficients generally fail to be log-concave---$\qfibonom{3+3}{3}$ is not log-concave for example---so the recent machinery developed for such results \cite{anari2018log, adiprasito2018hodge, branden2020lorentzian, chan2024log} is not applicable. 

We instead prove \Cref{thm:mainthm} algebraically in \Cref{sec:product_unimodality}. The cases $n = 1, 2$ are immediate, while $n = 3$ reduces to the following characterization of unimodality for a three-term product of $q$-analogs, proved in \Cref{sec:product_unimodality} as \Cref{lem:unimodal_iff_special_case}.

\begin{proposition}
    For $a, b, c \in \mathbb{N}$, the polynomial $T(q) = [a]_q [b]_q [c]_{q^2}$ is symmetric and unimodal if and only if $2c \leq a + b$, or $a$ or $b$ is even.
\end{proposition}

After establishing relevant background in \Cref{sec:background}, we also give an independent combinatorial proof for $n=2$ using the weighted path-domino tiling model of \cite{kustner} in \Cref{sec:comboproof}. We accomplish this by constructing equivalence classes of tilings which may be grouped to have symmetric and unimodal $q$-weights

\Cref{sec:conc} records further conjectures supported by computational evidence and develops the connection to the $q$-Fibonacci analog of the rational Catalan numbers, or the (rational) $q$-FiboCatalan numbers. In particular, we observe that \Cref{conj:mbynunimodal} would imply non-negativity of these polynomials, a fact stated but, to our knowledge, not previously proved.

\section{Background}\label{sec:background}

\subsection{Fibonacci Facts}

The Fibonacci numbers $F_n$ are defined by the recurrence $F_n=F_{n-1}+F_{n-2}$ for $n\geq2$ with initial conditions $F_0=0$ and $F_1=1$.

We collect a few well-known Fibonacci identities (see, for example, \cite{benjamin}) that will be useful to us in future sections.

\begin{proposition}\label{prop:identities}
    For all $m,n\geq0$, we have
    \begin{enumerate}
        \item\label{prop:id1} $F_1 + F_2+F_3 + \dots + F_n =F_{n+2}-1$.
        \item\label{prop:id2} $F_1 + F_3 + \dots + F_{2n-1}=F_{2n}$.
        \item\label{prop:id3}$F_2 + F_4 + \dots + F_{2n}=F_{2n+1}-1$.
        \item\label{prop:id4} If $m$ divides $n$, then $F_m$ divides $F_n$.
    \end{enumerate}
\end{proposition}

We will also be using Zeckendorf's theorem \cite{zeckendorf1972representations}: 

\begin{theorem}[\cite{zeckendorf1972representations}]\label{prop:zeckendorf}
    Every positive integer has a unique representation as a sum of non-consecutive Fibonacci numbers.
\end{theorem} 

\subsection{\texorpdfstring{$q$}{q}-Fibonomial Coefficients}

The Fibonomial coefficients were originally defined by Lucas \cite{lucas1878theorie}: for $m,n \in \mathbb{N}$, 
$$\fibonom{m+n}{n}:=\frac{F_{m+n}^!}{F_m^!F_n^!},$$
where $F_n ^!=F_n \cdot F_{n-1} \cdots F_{2}\cdot F_1$. Sagan--Savage \cite{sagan2009combinatorial} gave a combinatorial interpretation of these Fibonomial coefficients as enumerating certain path-domino tilings of an $m\times n$ rectangle, and \cite{kustner} introduced a $q$-weighting on this model that recovers the $q$-Fibonomial coefficients $\qfibonom{m+n}{n}$ as in \eqref{eq:qfibo}. We recall the construction.

Fix $m,n\in\mathbb{N}$. Choose a lattice path from $(0,0)$ to $(m,n)$, and tile the region above the path with squares and horizontal dominoes ($2\times 1$ rectangles) and the region below with squares and vertical dominoes ($1\times 2$ rectangles). We further require each tile immediately below a horizontal segment of the path to be a vertical domino. Assign weights to the tiles as follows:

\begin{center}

\begin{tikzpicture}[scale=0.7]
    \drawdomboard[-4.5]{1}{1}{}{}{};
    \node at (-4.9,0.5) {$w($};
    \node at (-2.7,0.5) {$)=1$,};
    \drawdomboard[-1]{2}{1}{2/1}{}{};
    \node at (-1.4,0.5) {$w($};
    \node at (2.3,0.5) {$)=q^{F_iF_j}$,};
    \drawdomboard[4.5]{1}{2}{}{1/2}{};
    \node at (4.1,1) {$w($};
    \node at (6.8,1) {$)=q^{F_iF_j}$,};
    \drawdomboard[9.5]{1}{2}{}{}{1/2};
    \node at (9.1,1) {$w($};
    \node at (11.9,1) {$)=q^{F_{i+1}F_j}$};
\end{tikzpicture}
\end{center}

where $(i,j)$ denotes the coordinates of the top right most corner of the tile, and the blue vertical domino represents the domino forced to lie immediately beneath the lattice path. Squares contribute weight 1 and are recoverable from the placement of dominoes, so we omit them from diagrams. 

The $q$-weight of the tiling $T$ is the product of the tiles in the path-domino tiling. We denote the set of all path-domino tilings of an $m\times n$ rectangle as $\mathcal{T}_{m,n}$. See \Cref{fig:pathdominotiling} for an example.

\begin{figure}
\begin{center}
\begin{tikzpicture}[scale=0.95]    
\draw[step=1cm] (0,0) grid (4,4);

\foreach \x/\y in {2/1} {
    \pgfmathsetmacro{\xa}{\x-2}
    \pgfmathsetmacro{\ya}{\y-1}
    \pgfmathsetmacro{\xb}{\x}
    \pgfmathsetmacro{\yb}{\y}
    \path[fill=gray!25, draw=black, thick]
        (\xa,\ya) rectangle (\xb,\yb);
}

\foreach \x/\y in {4/2} {
    \pgfmathsetmacro{\xa}{\x-1}
    \pgfmathsetmacro{\ya}{\y-2}
    \pgfmathsetmacro{\xb}{\x}
    \pgfmathsetmacro{\yb}{\y}
    \path[fill=gray!25, draw=black, thick]
        (\xa,\ya) rectangle (\xb,\yb);
}

\foreach \x/\y in {3/3,4/4} {
    \pgfmathsetmacro{\xa}{\x-1}
    \pgfmathsetmacro{\ya}{\y-2}
    \pgfmathsetmacro{\xb}{\x}
    \pgfmathsetmacro{\yb}{\y}
    \path[fill=blue!25, draw=black, thick]
        (\xa,\ya) rectangle (\xb,\yb);
}
    \node at (1,0.5) {\footnotesize$q^{F_2F_1}$};
    \node[font=\footnotesize] at (2.5,2) {$q^{F_4F_3}$};
    \node[font=\footnotesize] at (3.5,1) {$q^{F_4F_2}$};
    \node[font=\footnotesize] at (3.5,3) {$q^{F_5F_4}$};

    \node at (1,-0.5) {$F_1$};
    \node at (2,-0.5) {$F_2$};
    \node at (3,-0.5) {$F_3$};
    \node at (4,-0.5) {$F_4$};

    \node at (-0.5,1) {$F_1$}; 
    \node at (-0.5,2) {$F_2$}; 
    \node at (-0.5,3) {$F_3$}; 
    \node at (-0.5,4) {$F_4$}; 

    \draw[blue][very thick] (0,0) -- (2,0); 
    \draw[blue][very thick] (2,0) -- (2,3); 
    \draw[blue][very thick] (2,3) -- (3,3); 
    \draw[blue][very thick] (3,3) -- (3,4); 
    \draw[blue][very thick] (3,4) -- (4,4); 
    
\end{tikzpicture}
\end{center}
\caption{A path-domino tiling in $\mathcal{T}_{4,4}$ of weight $q^{25} = q^{F_2F_1}\cdot q^{F_4F_3}\cdot q^{F_5F_4}\cdot q^{F_4F_2}$. The lattice path is determined by the blue dominoes and is omitted in subsequent diagrams.}
\label{fig:pathdominotiling}
\end{figure}
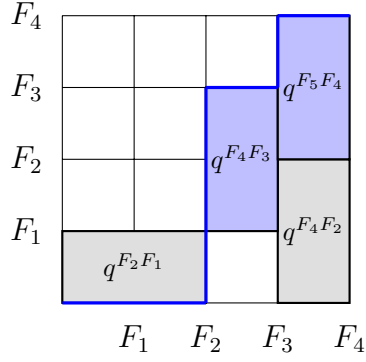

\begin{theorem}[{\cite[Theorem 2.4]{kustner}}]
For $m,n\in\mathbb{N}$, the $q$-analog of the Fibonomial coefficient can be computed as
$$\qfibonom{m+n}{n} = \sum_{T \in \mathcal{T}_{m,n}}w(T).$$
\end{theorem}

Notice that for $n = 1$, $\qfibonom{m+1}{1}$ is the sum of weights of the path-domino tilings of an $m\times 1$ rectangle (see \cite{sagan2009combinatorial}). The degree of the weight of each tiling is equal to a sum of non-consecutive Fibonacci numbers since only dominoes contribute weight, so \Cref{prop:zeckendorf} implies
\begin{equation}\label{eq:tilingstrip}
\qfibonom{m+1}{1} = \sum_{T \in \mathcal{T}_{m,1}}w(T)=[F_{m+1}]_q.
\end{equation}
The $n=1$ case of \Cref{thm:mainthm} follows immediately. 

\section{Symmetry and Unimodality of \texorpdfstring{$\qfibonom{m+2}{2}$}{[m+2][2]F}, Combinatorially}\label{sec:comboproof}

In this section we prove the case $n=2$ of \Cref{thm:mainthm} combinatorially.  We first define an equivalence relation on $\Tm$ and analyze its equivalence classes in \Cref{subsec:chain}, and then use this to prove symmetry and unimodality in \Cref{subsec:mby2combo}. 

\subsection{Partitioning \texorpdfstring{$\Tm$}{Tm2}} \label{subsec:chain}

    Define a map $\pi: \mathcal{T}_{m,2} \to \mathcal{T}_{m,2}$ as follows. 

    \begin{enumerate}[\normalfont(M1)]
    
    \item \label{move1} If there is a horizontal domino at $(2,1)$, remove it.

    
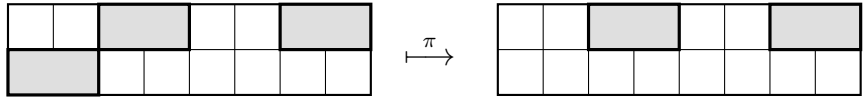
\begin{center}
        \begin{tikzpicture}[scale=0.6, baseline=(current bounding box.center)]
          \begin{scope}
            \def\W{8}
            \def\H{2}
        
            \draw[thick] (0,0) rectangle (\W,\H);
            \foreach \x in {1,...,7} \draw (\x,0) -- (\x,\H);
            \draw (0,1) -- (\W,1);
        
            \fill[gray!25] (2,1) rectangle (4,2);
            \draw[very thick] (2,1) rectangle (4,2);
        
            \fill[gray!25] (6,1) rectangle (8,2);
            \draw[very thick] (6,1) rectangle (8,2);
        
            \fill[gray!25] (0,0) rectangle (2,1);
            \draw[very thick] (0,0) rectangle (2,1);
          \end{scope}
        
          \node at (9.3,1) {$\overset{\pi}{\longmapsto}$};
        
          \begin{scope}[xshift=10.8cm]
            \def\W{8}
            \def\H{2}
        
            \draw[thick] (0,0) rectangle (\W,\H);
            \foreach \x in {1,...,7} \draw (\x,0) -- (\x,\H);
            \draw (0,1) -- (\W,1);
        
            \fill[gray!25] (2,1) rectangle (4,2);
            \draw[very thick] (2,1) rectangle (4,2);
        
            \fill[gray!25] (6,1) rectangle (8,2);
            \draw[very thick] (6,1) rectangle (8,2);
          \end{scope}
        \end{tikzpicture}
        \end{center}
    
    \item \label{move2} Otherwise, if there are any horizontal dominoes in the bottom row, shift the leftmost domino left by one tile and add as many dominoes directly to the left of it as possible.

        
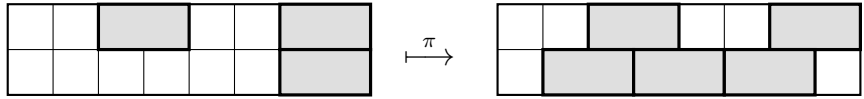
\begin{center}
            \begin{tikzpicture}[scale=0.6, baseline=(current bounding box.center)]
              \begin{scope}
                \def\W{8}
                \def\H{2}
            
                \draw[thick] (0,0) rectangle (\W,\H);
                \foreach \x in {1,...,7} \draw (\x,0) -- (\x,\H);
                \draw (0,1) -- (\W,1);
            
                \fill[gray!25] (2,1) rectangle (4,2);
                \draw[very thick] (2,1) rectangle (4,2);
            
                \fill[gray!25] (6,1) rectangle (8,2);
                \draw[very thick] (6,1) rectangle (8,2);
            
                \fill[gray!25] (6,0) rectangle (8,1);
                \draw[very thick] (6,0) rectangle (8,1);
              \end{scope}
            
              \node at (9.3,1) {$\overset{\pi}{\longmapsto}$};
            
              \begin{scope}[xshift=10.8cm]
                \def\W{8}
                \def\H{2}
            
                \draw[thick] (0,0) rectangle (\W,\H);
                \foreach \x in {1,...,7} \draw (\x,0) -- (\x,\H);
                \draw (0,1) -- (\W,1);
            
                \fill[gray!25] (2,1) rectangle (4,2);
                \draw[very thick] (2,1) rectangle (4,2);
            
                \fill[gray!25] (6,1) rectangle (8,2);
                \draw[very thick] (6,1) rectangle (8,2);
            
                \fill[gray!25] (1,0) rectangle (3,1);
                \draw[very thick] (1,0) rectangle (3,1);
            
                \fill[gray!25] (3,0) rectangle (5,1);
                \draw[very thick] (3,0) rectangle (5,1);
            
                \fill[gray!25] (5,0) rectangle (7,1);
                \draw[very thick] (5,0) rectangle (7,1);
              \end{scope}
            \end{tikzpicture}
        \end{center}

    \item \label{move3} Otherwise, if there are vertical dominoes, rotate the leftmost domino occupying $(k,2)$ to occupy $(k, 1)$ and add as many dominoes directly to the left of it as possible.

        
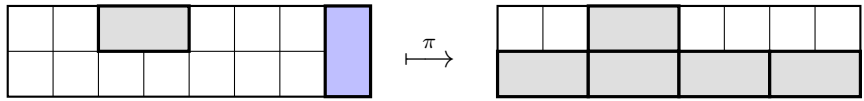
\begin{center}
            \begin{tikzpicture}[scale=0.6, baseline=(current bounding box.center)]
              \begin{scope}
                \def\W{8}
                \def\H{2}
            
                \draw[thick] (0,0) rectangle (\W,\H);
                \foreach \x in {1,...,7} \draw (\x,0) -- (\x,\H);
                \draw (0,1) -- (\W,1);
            
                \fill[gray!25] (2,1) rectangle (4,2);
                \draw[very thick] (2,1) rectangle (4,2);
            
                \fill[blue!25] (7,0) rectangle (8,2);
                \draw[very thick] (7,0) rectangle (8,2);
              \end{scope}
            
              \node at (9.3,1) {$\overset{\pi}{\longmapsto}$};
            
              \begin{scope}[xshift=10.8cm]
                \def\W{8}
                \def\H{2}
            
                \draw[thick] (0,0) rectangle (\W,\H);
                \foreach \x in {1,...,7} \draw (\x,0) -- (\x,\H);
                \draw (0,1) -- (\W,1);
            
                \fill[gray!25] (2,1) rectangle (4,2);
                \draw[very thick] (2,1) rectangle (4,2);
            
                \fill[gray!25] (0,0) rectangle (2,1);
                \draw[very thick] (0,0) rectangle (2,1);
            
                \fill[gray!25] (2,0) rectangle (4,1);
                \draw[very thick] (2,0) rectangle (4,1);
            
                \fill[gray!25] (4,0) rectangle (6,1);
                \draw[very thick] (4,0) rectangle (6,1);
            
                \fill[gray!25] (6,0) rectangle (8,1);
                \draw[very thick] (6,0) rectangle (8,1);
              \end{scope}
            \end{tikzpicture}

            \end{center}

    \item \label{move4} Otherwise, if the bottom row is completely empty, $\pi$ fixes the tiling.

    
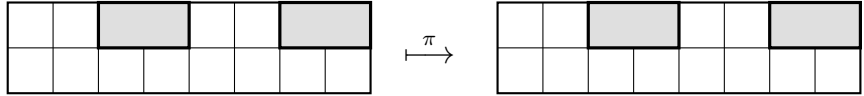
\begin{center}
        \begin{tikzpicture}[scale=0.6, baseline=(current bounding box.center)]
        \begin{scope}
          \def\W{8} 
          \def\H{2} 
        
          \draw[thick] (0,0) rectangle (\W,\H);
          \foreach \x in {1,...,7} \draw (\x,0) -- (\x,\H);
          \draw (0,1) -- (\W,1);
        
          \fill[gray!25] (2,1) rectangle (4,2);
          \draw[very thick] (2,1) rectangle (4,2);
        
          \fill[gray!25] (6,1) rectangle (8,2);
          \draw[very thick] (6,1) rectangle (8,2);
        \end{scope}
        \node at (9.3,1) {$\overset{\pi}{\longmapsto}$};
            
            \begin{scope}[xshift=10.8cm]
          \def\W{8} 
          \def\H{2} 
        
          \draw[thick] (0,0) rectangle (\W,\H);
          \foreach \x in {1,...,7} \draw (\x,0) -- (\x,\H);
          \draw (0,1) -- (\W,1);
        
          \fill[gray!25] (2,1) rectangle (4,2);
          \draw[very thick] (2,1) rectangle (4,2);

          \fill[gray!25] (6,1) rectangle (8,2);
          \draw[very thick] (6,1) rectangle (8,2);
        \end{scope}
        
        \end{tikzpicture}

    \end{center}
    \end{enumerate}

    We also define a partial inverse $\pi^*: \mathcal{T}_{m,2} \to \mathcal{T}_{m,2}$. 
    \begin{enumerate} 
        \item If $(2,1)$ does not contain a domino, add a horizontal domino there.
        \item Otherwise, if there exists a horizontal domino in the bottom row that can be shifted to the right, shift the leftmost domino that can be shifted to the right, and remove all the dominoes to the left of it. 
        \item Otherwise, if there exists a domino in the bottom row that can be rotated (i.e. there exists a horizontal domino occupying $(k, 1)$, and $(k,2)$ is empty, and there are no more horizontal dominoes to the right of that domino in the top row), rotate the rightmost domino and remove all the dominoes to the left of it on the bottom row. 
        \item If none of the above are possible, $\pi^*$ fixes the tiling.
    \end{enumerate}
    
\begin{remark}
    Note that $\pi(\pi^*(T)) = T$ if $\pi^*(T) \neq T$. Similarly, $\pi^*(\pi(T)) = T$ if $\pi(T) \neq T$.
\end{remark}
\begin{remark}\label{rmk:topunchanged}
    The top horizontal tiling remains unchanged under $\pi$ and $\pi^*$. 
\end{remark}    

We begin with the following observation about how $q$-degrees behave within a block.

\begin{lemma}\label{lem:decby1}
Let \(T\in\Tm\) with $\pi(T)\neq T$. Then $w(T) = q \cdot w(\pi(T))$.
\end{lemma}
\begin{proof}
    We consider each type of move separately.

    In move~\bmref{move1}, we obtain $\pi(T)$ by deleting a tile of weight $q^{F_1F_2} = q$, so $w(T) = q w(T')$.

   In move~\bmref{move2}, suppose the leftmost bottom row domino of $T$ occupies $(1,k)$. Then in order to obtain $\pi(T)$, we are replacing this tile of weight $q^{F_k}$ with the tiles of weight $q^{F_{k-1}}, q^{F_{k-3}},$ and so on up until $q^{F_3}$ or $q^{F_2}$, depending on the parity of $k$. If $k$ is odd, the product of the new terms is $q^{F_{k-1} + F_{k-3} + \ldots + F_3} = q^{F_k - 1}$ by \Cref{prop:identities}(2). Similarly, if $k$ is even, the product of the new terms is $q^{F_{k-1} + F_{k-3} + \ldots + F_2} = q^{F_k - 1}$ by \Cref{prop:identities}(3). In both cases, the total degree of the weight goes down by $1$.

    Finally, similarly to move~\bmref{move2}, move~\bmref{move3} replaces a tile of weight $q^{F_k}$ with the tiles of total weight $q^{F_{k-1} + F_{k-3} + \ldots} = q^{F_k - 1}$.

\end{proof}

Let $\sim$ denote the equivalence relation on $\Tm$ generated by $S \sim \pi(S)$. Equivalently, $S \sim T$ if and only if $S$ and $T$ have the same fixed point under iterated application of $\pi$ (which exists by \Cref{lem:decby1} and finiteness of $\Tm$). This partitions $\Tm$ into blocks. See \Cref{fig:fullexample,fig:singlechain} for examples.

\begin{figure}[t]
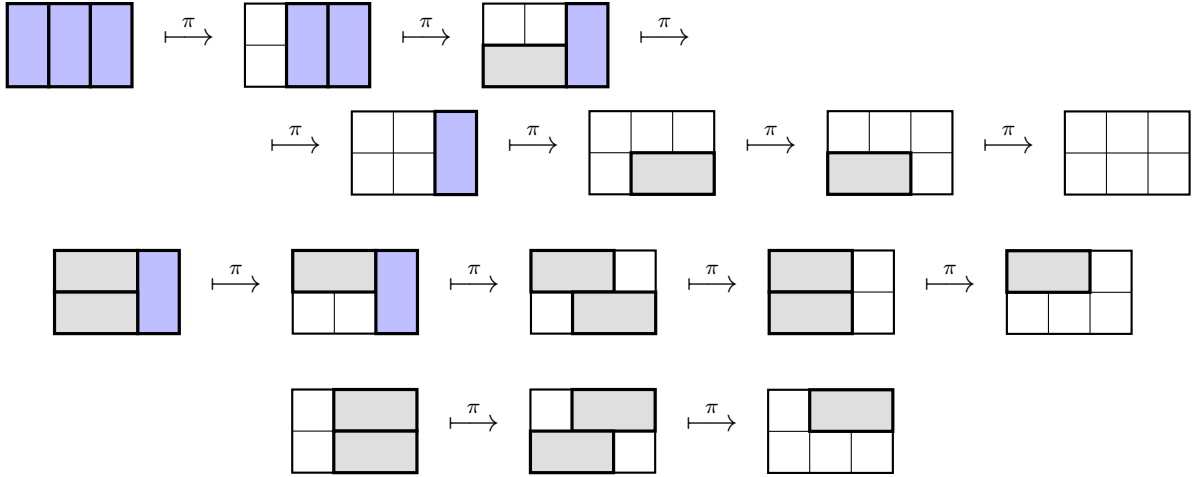

    \begin{center}
    \renewcommand{\arraystretch}{1.6}
    \begin{tabular}{c}
    
    \boardTwoByThree{ \vdom{0}\vdom{1}\vdom{2} }
    \;\ensuremath{\overset{\pi}{\longmapsto}}\;
    \boardTwoByThree{ \vdom{1}\vdom{2} }
    \;\ensuremath{\overset{\pi}{\longmapsto}}\;
    \boardTwoByThree{ \vdom{2}\hdom{0}{0} }\;
    \ensuremath{\overset{\pi}{\longmapsto}}\;\hspace{6.5cm}
    \\[1.5em]
    \hspace{3.5cm}
    \ensuremath{\overset{\pi}{\longmapsto}}\;
    \boardTwoByThree{ \vdom{2} }
    \;\ensuremath{\overset{\pi}{\longmapsto}}\;
    \boardTwoByThree{ \hdom{0}{1} }
    \;\ensuremath{\overset{\pi}{\longmapsto}}\;
    \boardTwoByThree{ \hdom{0}{0} }
    \;\ensuremath{\overset{\pi}{\longmapsto}}\;
    \boardTwoByThree{ }
    \\[2.5em]
    
    \boardTwoByThree{ \hdom{1}{0}\vdom{2}\hdom{0}{0} }
    \;\ensuremath{\overset{\pi}{\longmapsto}}\;
    \boardTwoByThree{ \hdom{1}{0}\vdom{2} }
    \;\ensuremath{\overset{\pi}{\longmapsto}}\;
    \boardTwoByThree{ \hdom{1}{0}\hdom{0}{1} }
    \;\ensuremath{\overset{\pi}{\longmapsto}}\;
    \boardTwoByThree{ \hdom{1}{0}\hdom{0}{0} }
    \;\ensuremath{\overset{\pi}{\longmapsto}}\;
    \boardTwoByThree{ \hdom{1}{0} }
    \\[2.5em]
    
    \boardTwoByThree{ \hdom{1}{1}\hdom{0}{1} }
    \;\ensuremath{\overset{\pi}{\longmapsto}}\;
    \boardTwoByThree{ \hdom{1}{1}\hdom{0}{0} }
    \;\ensuremath{\overset{\pi}{\longmapsto}}\;
    \boardTwoByThree{ \hdom{1}{1} }
    \\
    
    \end{tabular}
    \end{center}

\caption{Partitioning $\mathcal{T}_{3,2}$ into $F_3=3$ sets, cf. \Cref{lem:numchains}.}
\label{fig:fullexample}
\end{figure}

\begin{figure}[t]
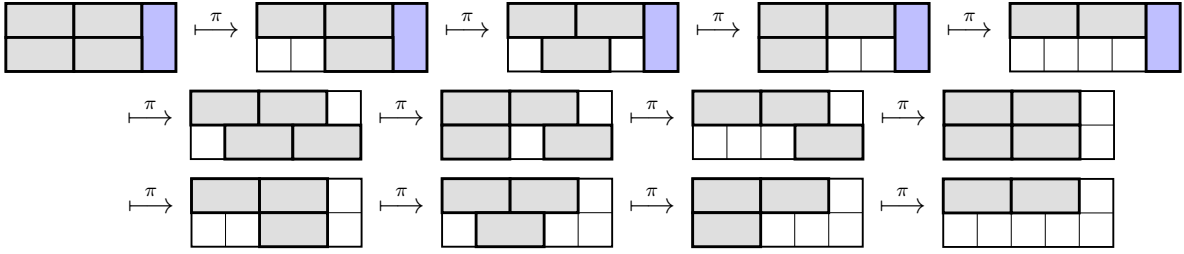


\begin{center}
\renewcommand{\arraystretch}{1}
\scalebox{.9}{
\begin{tabular}{c}

$\boardTwoByFive{ \TOPFIXED \vdom{4}\hdom{0}{0}\hdom{0}{2} }
\;\overset{\pi}{\longmapsto}\;
\boardTwoByFive{ \TOPFIXED \vdom{4}\hdom{0}{2} }
\;\overset{\pi}{\longmapsto}\;
\boardTwoByFive{ \TOPFIXED \vdom{4}\hdom{0}{1} }
\;\overset{\pi}{\longmapsto}\;
\boardTwoByFive{ \TOPFIXED \vdom{4}\hdom{0}{0} }
\;\overset{\pi}{\longmapsto}\;
\boardTwoByFive{ \TOPFIXED \vdom{4} }
$
\\[1.5em]

$\qquad
\overset{\pi}{\longmapsto}\;
\boardTwoByFive{ \TOPFIXED \hdom{0}{1}\hdom{0}{3} }
\;\overset{\pi}{\longmapsto}\;
\boardTwoByFive{ \TOPFIXED \hdom{0}{0}\hdom{0}{3} }
\;\overset{\pi}{\longmapsto}\;
\boardTwoByFive{ \TOPFIXED \hdom{0}{3} }
\;\overset{\pi}{\longmapsto}\;
\boardTwoByFive{ \TOPFIXED \hdom{0}{0}\hdom{0}{2} }
$
\\[1.5em]

$\qquad
\overset{\pi}{\longmapsto}\;
\boardTwoByFive{ \TOPFIXED \hdom{0}{2} }
\;\overset{\pi}{\longmapsto}\;
\boardTwoByFive{ \TOPFIXED \hdom{0}{1} }
\;\overset{\pi}{\longmapsto}\;
\boardTwoByFive{ \TOPFIXED \hdom{0}{0} }
\;\overset{\pi}{\longmapsto}\;
\boardTwoByFive{ \TOPFIXED }
$
\\

\end{tabular}
}
\end{center}
\caption{A block of the partition on $\mathcal{T}_{5,2}$ .}
\label{fig:singlechain}
\end{figure}

We introduce some terminology to refer to elements in the blocks.

\begin{definition}
A tiling $S \in \Tm$ is \textbf{minimal} if $\pi(S) = S$, and \textbf{maximal} if $\pi^*(S) = S$.   
\end{definition}

The following two properties follow from the definitions of $\pi$ and $\pi^*$. 

\begin{lemma}\label{rk:startstructure} The minimal and maximal tilings admit the following description.
\begin{enumerate}
    \item The tiling $S$ is minimal if and only if $S$ has no vertical dominoes and no horizontal dominoes on the bottom row. 
    \item Let $T$ be a tiling. Let $(\alpha_T,2)$ be the position of the rightmost horizontal domino in the top row of $T$. Then $T$ is maximal if and only if every column on the right of $\alpha_T$ is filled with a vertical domino and the bottom row contains as many horizontal dominoes as possible, leaving a gap at position $(1,1)$ if $\alpha_T$ is odd.
\end{enumerate}
\end{lemma}

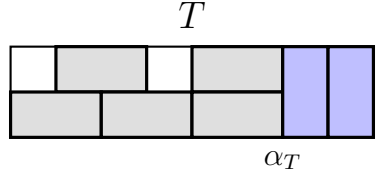
\begin{figure}[t]

\begin{center}
            \begin{tikzpicture}[scale=0.6, baseline=(current bounding box.center)]

                \def\W{8}
                \def\H{2}
            
                \draw[thick] (0,0) rectangle (\W,\H);
                \foreach \x in {1,...,7} \draw (\x,0) -- (\x,\H);
                \draw (0,1) -- (\W,1);

                \fill[gray!25] (4,1) rectangle (6,2);
                \draw[very thick] (4,1) rectangle (6,2);
                
                \fill[gray!25] (1,1) rectangle (3,2);
                \draw[very thick] (1,1) rectangle (3,2);

                \fill[gray!25] (0,0) rectangle (2,1);
                \draw[very thick] (0,0) rectangle (2,1);
            
                \fill[gray!25] (2,0) rectangle (4,1);
                \draw[very thick] (2,0) rectangle (4,1);
            
                \fill[gray!25] (4,0) rectangle (6,1);
                \draw[very thick] (4,0) rectangle (6,1);

                \fill[blue!25] (6,0) rectangle (7,2);
                \draw[very thick] (6,0) rectangle (7,2);
            
                \fill[blue!25] (7,0) rectangle (8,2);
                \draw[very thick] (7,0) rectangle (8,2);

                \node at (6,-0.5) {$\alpha_T$};
                \node[font=\large] at (4,2.7) {$T$};
            \end{tikzpicture}

            \end{center}                
            \caption{Maximal tiling in a partition}
\label{fig:alphat}
\end{figure}

This characterization yields the following description of the equivalence classes

\begin{lemma} \label{lem:uniquetoprow}
    $S\sim T$ if and only if the top horizontal tilings of $S$ and $T$ coincide.
\end{lemma}
\begin{proof}
    The forward direction follows from \Cref{rmk:topunchanged}. For the converse, suppose $S$ and $T$ have the same top-row horizontal tiling. Since $\mathcal{T}_{m,2}$ is finite, repeated application of $\pi$ to any tiling eventually reaches a fixed point, which by Lemma~\ref{rk:startstructure}(1) is uniquely determined by the tiling of the top row. Since $S$ and $T$ share the same top row, they reach the same fixed point, so $S \sim T$.
\end{proof}

By \Cref{lem:uniquetoprow}, blocks are in bijection with horizontal tilings of an $m$-strip, of which there are $F_{m+1}$.

\begin{corollary}\label{lem:numchains}
There are $F_{m+1}$ blocks in the partition of \(\mathcal{T}_{m,2}\) induced by $\sim$. 
\end{corollary}

\subsection{Weights of Minimal and Maximal Tilings} \label{subsec:mby2combo}

We describe the weights of minimal and maximal tilings. We begin with the minimal tilings.

\begin{corollary}\label{lem:uniquestart}
    The minimal tilings of $\Tm$ have weights $\{1,q,q^2,\dots,q^{F_{m+1}-1}\}$, each appearing exactly once.
\end{corollary}

\begin{proof}
By \Cref{rk:startstructure}, a minimal tiling has no bottom-row or vertical dominoes, so its weight is determined entirely by the horizontal tiling of its top row. Since the top-row dominoes of a minimal tiling occupy exactly the positions of a tiling of an $m$-strip, and $F_2=1$, the weights of the minimal tilings are the same as the weight of tilings of an $m$-strip. Thus, the result follows from $\sum_{T \in \mathcal{T}_{m,1}} w(T) = [F_{m+1}]_q $. 
\end{proof}

We prove the parallel statement for maximal tilings by first showing uniqueness of weights and then that they span the appropriate range.

\begin{lemma}\label{lem:uniquemax}
Let \(S\neq T\) be maximal tilings. Then \(w(S)\neq w(T)\).
\end{lemma}

\begin{proof}
Let $(\alpha_S, 2)$ and $(\alpha_T, 2)$ denote the positions of the rightmost horizontal dominoes in the top rows of $S$ and $T$ respectively. If $T$ does not have any dominoes in top row, we set $\alpha_T = 0$.

If $\alpha_S = \alpha_T$, then, by \Cref{rk:startstructure}, the bottom-row and vertical domino configurations of $S$ and $T$ are identical, so $S \neq T$  implies that their top-row tilings differ. Since distinct tilings of an $m$-strip have distinct weights by \eqref{eq:tilingstrip}, we conclude $w(S) \neq w(T)$.

Now suppose $\alpha_S \neq \alpha_T$. It  is enough to verify $\deg_qw(S) > \deg_qw(T)$ in the following case: let $\alpha_S +1 = \alpha_T$, let $S$ to be the  maximal tiling whose top row has a single horizontal domino ending at position  $\alpha_S = \alpha_T - 1$ (minimizing $\deg_qw(S)$ among tilings with parameter  $\alpha_S$), and $T$ to be the maximal tiling whose top row is filled with as many horizontal dominoes as possible ending at position $\alpha_T$ (maximizing $\deg_qw(T)$ among tilings with parameter $\alpha_T$), see \Cref{fig:st}. 

Using \Cref{rk:startstructure} and ignoring the common vertical  contributions to the right of $\alpha_T$, depending on the parity of $\alpha_T$, $\deg_qw(S)-\deg_qw(T)$ becomes: 
\begin{align*}
F_{\alpha_T+1}+F_{\alpha_T-1}+\bigl(F_{\alpha_T-1}+F_{\alpha_T-3}+\cdots+F_2\bigr)
\;&-\;
2\bigl(F_{\alpha_T}+F_{\alpha_T-2}+\cdots+F_3\bigr), & &\alpha_T \text{ odd},\\
F_{\alpha_T+1}+F_{\alpha_T-1}+\bigl(F_{\alpha_T-1}+F_{\alpha_T-3}+\cdots+F_3\bigr)
\;&-\;
2\bigl(F_{\alpha_T}+F_{\alpha_T-2}+\cdots+F_2\bigr) , & &\alpha_T \text{ even}.
\end{align*}
By \Cref{prop:identities}(2,3), both expressions evaluate to $(2F_{\alpha_T+1} - 1) - (2F_{\alpha_T+1} - 2)>0$.
\end{proof}

\begin{figure}[t]
        \begin{center}
            \begin{tikzpicture}[scale=0.6, baseline=(current bounding box.center)]
              \begin{scope}
                \def\W{8}
                \def\H{2}
            
                \draw[thick] (0,0) rectangle (\W,\H);
                \foreach \x in {1,...,7} \draw (\x,0) -- (\x,\H);
                \draw (0,1) -- (\W,1);
            
                \fill[gray!25] (3,1) rectangle (5,2);
                \draw[very thick] (3,1) rectangle (5,2);

                \fill[gray!25] (1,0) rectangle (3,1);
                \draw[very thick] (1,0) rectangle (3,1);
            
                \fill[gray!25] (3,0) rectangle (5,1);
                \draw[very thick] (3,0) rectangle (5,1);
            
                \fill[blue!25] (5,0) rectangle (6,2);
                \draw[very thick] (5,0) rectangle (6,2);
                \fill[blue!25] (6,0) rectangle (7,2);
                \draw[very thick] (6,0) rectangle (7,2);
                \fill[blue!25] (7,0) rectangle (8,2);
                \draw[very thick] (7,0) rectangle (8,2);

                \node at (5,-0.5) {$\alpha_S$};
                \node at (6,-0.5) {$\alpha_T$};

                \node[font=\large] at (4,2.7) {$S$};
              \end{scope}

              \begin{scope}[xshift=10.8cm]
                \def\W{8}
                \def\H{2}
            
                \draw[thick] (0,0) rectangle (\W,\H);
                \foreach \x in {1,...,7} \draw (\x,0) -- (\x,\H);
                \draw (0,1) -- (\W,1);

                \fill[gray!25] (4,1) rectangle (6,2);
                \draw[very thick] (4,1) rectangle (6,2);
                
                \fill[gray!25] (2,1) rectangle (4,2);
                \draw[very thick] (2,1) rectangle (4,2);
                
                \fill[gray!25] (0,1) rectangle (2,2);
                \draw[very thick] (0,1) rectangle (2,2);
            
                \fill[gray!25] (0,0) rectangle (2,1);
                \draw[very thick] (0,0) rectangle (2,1);
            
                \fill[gray!25] (2,0) rectangle (4,1);
                \draw[very thick] (2,0) rectangle (4,1);
            
                \fill[gray!25] (4,0) rectangle (6,1);
                \draw[very thick] (4,0) rectangle (6,1);

                \fill[blue!25] (6,0) rectangle (7,2);
                \draw[very thick] (6,0) rectangle (7,2);
            
                \fill[blue!25] (7,0) rectangle (8,2);
                \draw[very thick] (7,0) rectangle (8,2);
                \node at (5,-0.5) {$\alpha_S$};
                \node at (6,-0.5) {$\alpha_T$};
                \node[font=\large] at (4,2.7) {$T$};
              \end{scope}
            \end{tikzpicture}
\caption{If $\alpha_S<\alpha_T$, then $\deg_qw(S)>\deg_qw(T)$.}
\label{fig:st}
            \end{center}
\end{figure}
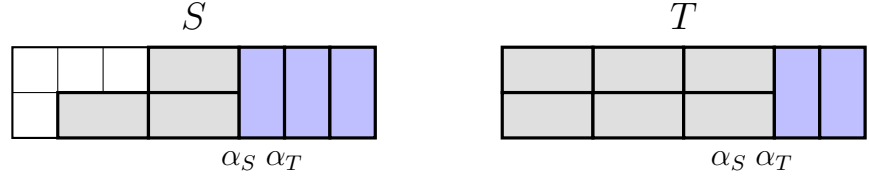

\begin{lemma}\label{lem:minmaxend}
    Fix \(m\ge 1\). The smallest possible \(q\)-degree of a maximal tiling in $\Tm$ is $F_{m + 2} - 1$. The largest possible \(q\)-degree of a maximal tiling in $\Tm$ is $F_{m + 3} - 2$.

\end{lemma}

\begin{proof} 

Let $T_{\min}$ denote the maximal tiling with smallest possible degree. From the proof of \Cref{lem:uniquemax}, to minimize the \(q\)-degree of a maximal tiling, the rightmost horizontal domino in the top row of $T_{\min}$ must occupy $(m,2)$. Thus, the smallest $q$-degree at the end of a chain occurs when the top row contains exactly one horizontal domino which occupies $(m,2)$.

In that case, we obtain
\[
\deg_q w(T_{\min})=
\begin{cases}
2F_{m}+F_{m-2}+\cdots+F_2, & m \text{ even},\\[4pt]
2F_{m } +F_{m-2}+\cdots+F_3, & m \text{ odd}.
\end{cases}
\]
Using \Cref{prop:identities}(2,3), we obtain \(\deg_q w(T_{\min}) = F_{m + 1} + F_m - 1 = F_{m + 2} - 1\).

Let $T_{\max}$ be the tiling consisting solely of vertical dominoes. By \Cref{rk:startstructure}, $T_{\max}$ is a maximal tiling. We have
\[
    \deg_qw(T_{\max}) =F_{m+1}+\dots+F_3+F_2 = F_{m+3}-2,
\]
by \Cref{prop:identities}(1). This is clearly the largest $q$-degree of a maximal tiling as it is the largest $q$-degree of any tiling in $\Tm$.
\end{proof}

\subsection{Proof of Symmetry and Unimodality}

We are now ready to find $\qfibonom{m+2}{2}$.

\begin{theorem}\label{thm:closedform}
Fix \(m\ge 1\) and write
\[
\qfibonom{m+2}{2} = \sum_{k=0}^{F_{m+3}-2} a_k\, q^k,
\quad\text{where}\quad
a_k = \begin{cases} k+1 & 0 \le k \le F_{m+1}-1, \\ F_{m+1} & F_{m+1} \le k \le F_{m+2}-2, \\ F_{m+3}-k-1 & F_{m+2}-1 \le k \le F_{m+3}-2. \end{cases}
\]

\end{theorem}

\begin{proof}
By \Cref{lem:numchains}, the moves $\pi$ partition \(\mathcal{T}_{m,2}\) into \(F_{m+1}\) disjoint blocks.
By \Cref{lem:decby1}, each block consists of tilings whose \(q\)-degrees form a
consecutive interval of integers, and each degree in that interval appears exactly once in that block.

Write the blocks as \(B_1,\dots,B_{F_{m+1}}\), and let \(\ell_i\) (resp.\ \(r_i\)) be the degree of the minimal
(resp.\ maximal) element of \(B_i\). Then
\[
\sum_{T\in B_i} w(T) \;=\; q^{\ell_i}+q^{\ell_i+1}+\cdots+q^{r_i},
\]
so the coefficient \(a_k\) equals the number of blocks whose degree-interval contains \(k\), i.e.
\[
a_k \;=\; \#\{\,i:\ \ell_i\le k\le r_i\,\}.
\]

By \Cref{lem:uniquestart}, we obtain 
$$ \{\ell_1,\dots,\ell_{F_{m+1}}\}=\{0,1,\dots,F_{m+1}-1\}.$$

By \Cref{lem:uniquemax,lem:minmaxend}, we obtain

$$ \{r_1,\dots,r_{F_{m+1}}\}=\{F_{m+2}-1,\dots,F_{m+3}-2\}.$$

The form of the polynomial follows. 
\end{proof}

From this, we have \Cref{thm:mainthm} for $n=2$.

\begin{corollary} 
    The coefficient sequence of the polynomial $\qfibonom{m+2}{2}$ is symmetric and unimodal.
\end{corollary}

\begin{remark}
    The blocks of the partition induced by $\sim$ and our moves $\pi$ can be interpreted as a chain in a \textbf{nearly} symmetric saturated chain decomposition of the poset $(\Tm,\leq)$ where $S\leq T$ if $\deg_qw(S)\leq \deg_qw(T)$. Note that these chains are not symmetric as we do not generally have $\ell_i = F_{m+3}-2-r_i$.
\end{remark}

\section{Unimodality of \texorpdfstring{$\qfibonom{m+3}{3}$}{[m+3][3]F}, Algebraically}\label{sec:product_unimodality}

In this section, we consider an algebraic approach. We begin by proving various facts about $q$-analogs, and then apply these facts to the $q$-Fibonomial coefficients in \Cref{subsec:n=3} to prove symmetry and unimodality. 

\subsection{Key Properties of \texorpdfstring{$q$}{q}-analogs}\label{subsec:qlemmas}

\begin{lemma}
    \label{lem:prod-of-q-analogs-unimodal}
    Let $a\leq b$ be two positive integers. Then,
    $$[a]_q[b]_q=\sum_{k=0}^{a+b-2}c_kq^k, \text{ where }c_k=\begin{cases}
        k+1, & 0\leq k \leq a-1, \\
        a, & a \leq k \leq b-1, \\
        a+b-1-k, & b \leq k \leq a+b-2.
    \end{cases}$$
    In particular, $[a]_q[b]_q$ is symmetric and unimodal.
\end{lemma}
\begin{proof}

We express each polynomial as a finite geometric series:
\[
[a]_q[b]_q=\frac{1-q^a}{1-q}\cdot\frac{1-q^b}{1-q}
=(1-q^a-q^b+q^{a+b})(1-q)^{-2}.
\]
Using the power series expansion $(1-q)^{-2} = \sum_{j=0}^{\infty} (j+1)q^j$, we distribute the numerator:
    \[
    [a]_q[b]_q
    =\sum_{j=0}^{\infty} (j+1)q^j
    -\sum_{j=0}^{\infty} (j+1)q^{j+a}
    -\sum_{j=0}^{\infty} (j+1)q^{j+b}
    +\sum_{j=0}^{\infty} (j+1)q^{j+a+b}.
    \]
    Let $[a]_q[b]_q=\sum_k c_kq^k$. The coefficient $c_k$ of $q^k$ is found by shifting the indices and applying the indicator function $\mathbb{I}$:
    \[
    c_k
    =(k+1)
    -\mathbb{I}_{k \ge a}(k-a+1)
    -\mathbb{I}_{k \ge b}(k-b+1)
    +\mathbb{I}_{k \ge a+b}(k-a-b+1).
    \] 
    This simplifies to the statement above.
\end{proof}

\begin{lemma}\label{lem:unimodal_an}
    Suppose $a, b, {\pow}$ are three positive integers such that ${\pow}\,|\, a$. Then $[a]_q[b]_{q^{{\pow}}}$ is symmetric and unimodal.

\end{lemma}
\begin{proof}
    Let $a = k {\pow}$.
    First, notice that
    $$[a]_q = \dfrac{1 - q^{k {\pow}}}{1 - q} = \dfrac{1 - q^{\pow}}{1 - q} \cdot \dfrac{1 - (q^{\pow})^k}{1 - q^\pow} = (1+q + \ldots + q^{{\pow}-1}) [k]_{q^{\pow}}.$$
    Therefore 
    $$[a]_q [b]_{q^{\pow}} = (1 + q + \ldots + q^{{\pow}-1})[k]_{q^{\pow}} [b]_{q^{\pow}}.$$

    Denote $[k]_q[b]_q = \sum_{i = 0}^N c_i q^i$. Then 
    \begin{align*}
        [a]_q [b]_q &= (1 + q + \ldots + q^{{\pow}-1})[k]_{q^{\pow}} [b]_{q^{\pow}} = (1 + q + \ldots + q^{{\pow}-1})\sum_{i=0}^N c_i q^{i{\pow}} = \sum_{j = 0}^{{\pow}-1} \sum_{i=0}^N c_i q^{i{\pow} + j}\\ &= c_0 + c_0q + \ldots + c_0 q^{{\pow}-1} + c_1 q^{\pow} + c_1 q^{{\pow}+1} + \ldots + c_{N} q^{N{\pow} + {\pow}-1}.
    \end{align*}
    By \Cref{lem:prod-of-q-analogs-unimodal}, the coefficient sequence $\{c_k\}$ is symmetric and unimodal.
\end{proof}

\begin{corollary} \label{cor:divisionunimodal}
    Let $a_1,...,a_k,b, \pow$ be positive integers. Then if $r$ divides $a_i$ for some $1\leq i \leq k$, the product
    \[
        [a_1]_q\dots[a_k]_q[b]_{q^{\pow}}
    \]
    is symmetric and unimodal.
\end{corollary}
\begin{proof}
    By \Cref{lem:unimodal_an}, $[a_i]_q[b]_{q^{\pow}}$ is symmetric and unimodal. As each $[a_j]_q$ is symmetric and unimodal for $1\leq j \leq k$, $[a_1]_q\dots[a_k]_q[b]_{q^{\pow}}$ is also symmetric and unimodal (see, for example, \cite[Observation 2]{zeilberger1989one}).
\end{proof}

Before we state the main algebraic lemma, we need another technical lemma.

\begin{lemma}
    Suppose $a,b,c\in\mathbb{Z}$ with $a \leq b$ and $a,b$ odd. Define
    $$A(k) = \# \left \{ \ell \in [0, c - 1] \, \Bigm| \, \frac{k - a + 2}{2} \le \ell  \le \frac{k + 1}{2}, \; \ell \in \Z_{\geq 0} \right \},$$
    and 
    $$B(k) =  \# \left \{ \ell \in [0, c - 1] \, \Bigm|\, \frac{k - a - b + 2}{2} \le \ell \le \frac{k - b + 1}{2}, \; \ell \in \Z_{\geq0} \right \}.$$
    Then:
    \begin{enumerate}
        \item If $2c \leq a+b$, then $A(k) \geq B(k)$ for every $k \leq c + \frac{a+b}{2} - 3$.
        \item If $2c > a+b$, then $A(a+b-2) < B(a+b-2)$.
    \end{enumerate}
    \label{lem:Ak-greater-than-Bk}
\end{lemma}

\begin{proof}
Assume $2c \leq a+b$ and fix $k \leq c + \tfrac{a+b}{2} - 3$.
 
Since $2c \leq a+b$, we have $k \leq a+b-3$, so $\tfrac{k-a-b+2}{2} \leq -\tfrac{1}{2}$ and the lower bound on $\ell$ in $B(k)$ reduces to $0 \leq \ell$. Since $a \leq b$, we have $k \leq c + b - 3$, so $\tfrac{k-b+1}{2} \leq c - 1$ and the upper bound reduces to $\ell \leq \tfrac{k-b+1}{2}$. Hence
\begin{align}
    B(k)&=\# \left \{ \ell \in \Z \Bigm| 0 \le \ell \le \frac{k - b + 1}{2}\right \} \nonumber\\
    &=\begin{cases} 0, & k < b-1, \\ \left\lfloor \dfrac{k-b+1}{2} \right\rfloor + 1, & b-1 \leq k \leq c + \dfrac{a+b}{2} - 3. \label{eq:Bk-formula} \end{cases}
\end{align}
For $k < b-1$, then, $B(k) = 0 \leq A(k)$.
 
Now assume $k \geq b-1$. Then $\tfrac{k-a+2}{2} \geq \tfrac{1}{2}$, so the lower bound on $\ell$ in $A(k)$ reduces to $\tfrac{k-a+2}{2} \leq \ell$, while the upper bound is $\min\bigl(\tfrac{k+1}{2}, c-1\bigr)$, with $\tfrac{k+1}{2} \leq c-1$ equivalent to $k \leq 2c-3$. Since $a$ is odd, both endpoints $\tfrac{k-a+2}{2}$ and $\tfrac{k+1}{2}$ are integers when $k$ is odd and half-integers when $k$ is even. Thus
\begin{align}
A(k) &= \# \left \{ \ell \in \Z \, \Bigm|  \frac{k - a + 2}{2} \le \ell  \le \min \left(\frac{k + 1}{2},c-1\right) \right \} \nonumber \\
&= \begin{cases} 
    \dfrac{a-1}{2}, & b-1 \leq k \leq 2c-3,\ k \text{ even}, \\[5pt] 
    \dfrac{a+1}{2}, & b-1 \leq k \leq 2c-3,\ k \text{ odd}, \\[5pt]
    \left\lfloor \dfrac{a+2c-k-2}{2} \right\rfloor, & 2c-2 \leq k \leq c + \dfrac{a+b}{2} - 3. \label{eq:Ak-formula} 
\end{cases}
\end{align}
 
For $b-1 \leq k \leq 2c-3$, we have $k - b + 1 \leq a - 2$, so $B(k) \leq \lfloor \tfrac{a-2}{2} \rfloor + 1 = \tfrac{a-1}{2} \leq A(k)$, where the final equality uses $a$ odd.
 
For $2c-2 \leq k \leq c + \tfrac{a+b}{2} - 3$, $A$ is non-increasing and $B$ is non-decreasing, so it suffices to verify $A(c + \tfrac{a+b}{2} - 3) \geq B(c + \tfrac{a+b}{2} - 3)$. Then \eqref{eq:Bk-formula} and \eqref{eq:Ak-formula} give
\[
    A(c + \tfrac{a+b}{2} - 3) = \left\lfloor \tfrac{c + \tfrac{a-b}{2}+1}{2} \right\rfloor \geq \left\lfloor \tfrac{c + \tfrac{a-b}{2}}{2} \right\rfloor = B(c + \tfrac{a+b}{2} - 3).
\]
This proves~(1).

For~(2), suppose $2c > a+b$. We have 
$$B(a+b-2)= \# \left \{ \ell \in [0, c - 1] \, \Bigm|\, 0 \le \ell \le \frac{a -1}{2}, \; \ell \in \Z_{\geq 0} \right \}=\frac{a+1}{2}$$
since $a$ is odd. 
$$A(a+b-2)=\# \left \{ \ell \in [0, c - 1] \, \Bigm| \, \frac{b}{2} \le \ell  \le \frac{a+b-1}{2}, \; \ell \in \Z_{ \geq 0} \right \}=\frac{a-1}{2}$$
since $b, a+b-1$ are odd and $\frac{a+b-1}{2} < c-1$. We have $A(a+b-2)<B(a+b-2)$.

\end{proof}

\begin{proposition}
\label{lem:unimodal_iff_special_case}
    For any $a, b,c \in \mathbb{N}$, $T(q)=[a]_q[b]_q[c]_{q^2}$ is symmetric and unimodal if and only if $2c \leq a+b$ or $a$ or $b$ is even.
\end{proposition}
\begin{proof}
Assume without loss of generality $a \leq b$. If $a$ or $b$ is even, this follows from \Cref{cor:divisionunimodal}. 

Assume $a$ and $b$ are odd. Since $T(q)$ is a product of symmetric polynomials, it is symmetric, and it is unimodal if and only if $[q^{k+1}] T(q) - [q^k] T(q) \geq 0$ for every $0 \leq k \leq \frac{\deg T}{2} - 1 = c + \frac{a+b}{2} - 3$.

Write
\[
T(q) = \sum_{\ell = 0}^{c-1} T_\ell(q), \qquad T_\ell(q) := q^{2\ell}[a]_q[b]_q,
\]
so that $[q^{k+1}] T(q) - [q^k] T(q) = \sum_\ell \bigl([q^{k+1}] T_\ell(q) - [q^k] T_\ell(q)\bigr)$. Applying \Cref{lem:prod-of-q-analogs-unimodal} to the coefficients of $T_\ell$, we get
\[
[q^{k+1}] T_\ell(q) - [q^k] T_\ell(q) =
\begin{cases}
1, & 2\ell - 1 \leq k \leq 2\ell + a - 2, \\
-1, & 2\ell + b - 1 \leq k \leq 2\ell + a + b - 2, \\
0, & \text{otherwise.}
\end{cases}
\]

Define
\begin{align*}
A(k) &:= \#\{\ell \in [0, c-1] : 2\ell - 1 \leq k \leq 2\ell + a - 2\}\\
&= \#\left\{\ell \in [0, c-1] : \tfrac{k - a + 2}{2} \leq \ell \leq \tfrac{k + 1}{2}\right\}, \\
B(k) &:= \#\{\ell \in [0, c-1] : 2\ell + b - 1 \leq k \leq 2\ell + a + b - 2\} \\
&= \#\left\{\ell \in [0, c-1] : \tfrac{k - a - b + 2}{2} \leq \ell \leq \tfrac{k - b + 1}{2}\right\},
\end{align*}
so that $[q^{k+1}] T(q) - [q^k] T(q) = A(k) - B(k)$. Both directions of the lemma now follow from \Cref{lem:Ak-greater-than-Bk}.
\end{proof}

\subsection{Proof of \texorpdfstring{\Cref{thm:mainthm}}{Theorem \ref{thm:mainthm}}}\label{subsec:n=3}

We now prove unimodality of $\qfibonom{m+1}{1}, \qfibonom{m+2}{2},\qfibonom{m+3}{3}$. 

\begin{proof}[Proof of \Cref{thm:mainthm}]
The case $n=1$ is immediate, since $\qfibonom{m+1}1=[F_{m+1}]_q$. For $n=2$, $\qfibonom{m+2}{2} = [F_{m+1}]_q[F_{m+2}]_q$, and the result follows from \Cref{lem:prod-of-q-analogs-unimodal}.

It remains to treat $n=3$. By definition of factorials and since $[F_3]_q^! = 1+q$,
\[
\qfibonom{m+3}{3}
=\frac{[F_{m+3}]_q^!}{[F_3]_q^!\,[F_m]_q^!}
=\frac{[F_{m+3}]_q [F_{m+2}]_q [F_{m+1}]_q}{1+q}.
\]
Among the three consecutive Fibonacci numbers $F_{m+1}, F_{m+2}, F_{m+3}$, exactly one is even by \Cref{prop:identities}(4); call this term $E$ and the remaining (odd) terms $A \leq B$. Then
\[
\qfibonom{m+3}{3} = [A]_q [B]_q \cdot \frac{[E]_q}{1+q} = [A]_q [B]_q \cdot \left[\tfrac{E}{2}\right]_{q^2}.
\]
Since $2\cdot \frac{E}2 \leq A + B$, \Cref{lem:unimodal_iff_special_case} gives that $\qfibonom{m+3}{3}$ is symmetric and unimodal.
\end{proof}

\section{Future work: More on symmetry, \texorpdfstring{$q$}{q}-analogs, and Fibocatalan numbers}\label{sec:conc}

The methods developed in this paper suggest several further directions beyond the cases proved thus far. In particular, our combinatorial and algebraic proofs of unimodality raise related questions about symmetry, products of $q$-analogs, and $q$-FiboCatalan numbers.

\subsection{Combinatorial Proof of Symmetry}
We've proven the unimodality of $\qfibonom{m+n}{n}$ for $n\leq3$ and verified it computationally\footnote{Source code: \href{https://github.com/brendanconnelly/Fibonomial-Conjecture-Tests}{https://github.com/brendanconnelly/Fibonomial-Conjecture-Tests}.\label{fn:code}} for $m+n\leq20$ and $m=n\leq16$,  but proving for all $m,n$ remains an open problem. 

The symmetry of the coefficient sequence of $\qfibonom{m+n}{n}$ is algebraically clear. However,  with the weighted path-domino tiling model, this fact is combinatorially non-trivial. We find a combinatorial explanation in \Cref{sec:comboproof} for $n=2$, and leave the general case as an open problem.

\begin{problem}
    Give a combinatorial proof of the symmetry of the coefficient sequence of $\qfibonom{m+n}{n}$ for all $m,n\geq3$. 
\end{problem}

Similarly, symmetry with respect to $m$ and $n$ is algebraically clear, but it lacks a combinatorial interpretation and is nontrivial from the path-domino tiling model.

\begin{problem}
    Give a combinatorial proof of the equality $\qfibonom{m+n}{n}=\qfibonom{m+n}{m}$ for all $m,n$.
\end{problem}

\subsection{More on \texorpdfstring{$q$}{q}-analogs}

In proving the symmetry and unimodality of $\qfibonom{m+3}{3}$, we encounter several lemmas regarding $q$-analogs, for example \Cref{lem:unimodal_iff_special_case}. We attempt to generalize some of these results below. In \Cref{lem:unimodal_an}, we proved the simplest case needed to prove \Cref{lem:unimodal_iff_special_case}. We prove an analogous proposition below.

\begin{proposition}\label{prop:unimodal}
    Let $a,b,\pow\in\mathbb{Z}_{>0}$. Then $[a]_q[b]_{q^{\pow}}$ is unimodal if and only if $a \geq \pow(b-1)$ or $\pow \,|\, a$. 
\end{proposition}
\begin{proof}

If $\pow \,|\, a$, then $[a]_q[b]_{q^{\pow}}$ is unimodal by \Cref{lem:unimodal_an}. Assume $\pow \nmid a$; we show $[a]_q[b]_{q^{\pow}}$ is unimodal if and only if $a \geq \pow(b-1)$.

Write
\[
P(q) := [a]_q[b]_{q^{\pow}} = \sum_{i=0}^{a-1}\sum_{j=0}^{b-1} q^{i+{\pow}j} = \sum_{t=0}^{D} c_t q^t, \qquad D := (a-1)+{\pow}(b-1).
\]
Since $[a]_q$ and $[b]_{q^{\pow}}$ are symmetric of degrees $a-1$ and $r(b-1)$, $P$ is symmetric of degree $D$.

Suppose first that $a \geq {\pow}(b-1)$, so $\lfloor \frac{D}{2} \rfloor \leq a-1$. For $0 \leq t \leq a-1$, the constraint $i + {\pow}j = t$ with $0 \leq i \leq a-1$ forces only ${\pow}j \leq t$, so
\[
c_t = \#\{j \in \{0,\dots,b-1\} : {\pow}j \leq t\} = \min\!\left(b-1, \lfloor \frac{t}{{\pow}} \rfloor\right) + 1,
\]
which is weakly increasing in $t$. Symmetry then gives unimodality.

Conversely, suppose $a < \pow(b-1)$, so $a \leq \frac{D}{2}$; we show $c_a < c_{a-1}$. For $c_{a-1}$: any $j \geq 0$ with $\pow j \leq a-1$ satisfies $j < b$ and $i := a-1-\pow j \in [0,a-1]$, so $c_{a-1} = \lfloor \frac{a-1}{\pow} \rfloor + 1$. For $c_a$: we need $i = a - \pow j \leq a-1$, forcing $j \geq 1$, so $c_a = \#\{j \geq 1 : \pow j \leq a\} = \lfloor \frac{a}{\pow}\rfloor = \lfloor \frac{a-1}{\pow} \rfloor$, using $\pow \nmid a$. Hence $c_a < c_{a-1}$.
\end{proof}

We also conjecture the following as a generalization of \Cref{cor:divisionunimodal}, \Cref{lem:unimodal_iff_special_case}, and \Cref{prop:unimodal}.

\begin{conjecture}\label{conj:unimodalitylemma}
Let $\pow \geq2, k \geq 1$ and let $a_1,\ldots, a_k,b$ be positive integers. Then if $\pow \mid a_i$ for any $1 \leq i \leq k$ or
\[
b \leq 1+\sum_{i=1}^{k}\left\lfloor\frac{a_i}{\pow}\right\rfloor,
\] 
the polynomial $[a_1]_q\dots[a_k]_q[b]_{q^{\pow}}$ is unimodal. Moreover, if $k \leq 3$ or $\pow \leq 3$, this condition is also necessary.
\end{conjecture}

We have computationally verified \Cref{conj:unimodalitylemma} for $k \leq 5, \pow \leq 6$ and $\max\{a_i,b\} \leq 15$. In general, the condition is not necessary: $([3]_q)^4 [2]_{q^4}$ is unimodal despite violating the inequality $b \leq 1 + \sum_{i=1}^{k}\lfloor \frac{a_i}{\pow} \rfloor$.

We also note that \Cref{conj:unimodalitylemma} does not directly imply unimodality of $\qfibonom{m+n}{n}$, because the factorization of $\qfibonom{m+n}{n}$ for $n\geq 4$ as a product of $q$-analogs does not match the form of $[a_1]_q\dots[a_{k}]_q[b]_{q^{\pow}}$ in the conjecture.

\subsection{\texorpdfstring{$q$}{q}-FiboCatalan Numbers}

For relatively prime integers $m, n \in \mathbb{Z}_{>0}$, the rational Catalan number
\[
    C_{m,n} = \frac{1}{m+n}\binom{m+n}{n}
\]
counts the lattice paths from $(0,0)$ to $(m,n)$ that stay weakly above the main diagonal of the $m \times n$ rectangle \cite{bizley1954derivation}. The Fibonacci analog, or the FiboCatalan number, is defined as
\[
    \fibocat{m}{n} := \frac{1}{F_{m+n}}\fibonom{m+n}{n} = \frac{F_{m+n-1}^!}{F_m^!\, F_n^!}.
\]
We then define the \emph{$q$-FiboCatalan number}
\[
    \qfibocat{m}{n} := \frac{1}{[F_{m+n}]_q}\qfibonom{m+n}{n} = \frac{[F_{m+n-1}]^!_q}{[F_m]^!_q\, [F_n]^!_q}.
\]
Bergeron--Ceballos--K\"ustner proved that this rational expression is in fact a polynomial with integral coefficients when $\gcd(m,n) \in \{1,2\}$, and conjectured non-negativity based on computational evidence.
 
\begin{proposition}[{\cite[Proposition 4.2]{kustner}}]\label{prop:fibocatintegral}
    If $\gcd(m,n) \in \{1,2\}$, then $\qfibocat{m}{n} \in \mathbb{Z}[q]$.
\end{proposition}
 
We show that \Cref{conj:mbynunimodal} implies this observed non-negativity.
 
\begin{proposition}\label{prop:fibocat-nonneg}
Let $m, n \geq 1$ with $\gcd(m,n) \in \{1,2\}$. If $\qfibonom{m+n}{n}$ is unimodal, then $\qfibocat{m}{n} \in \mathbb{Z}_{\geq 0}[q]$.
\end{proposition}
 
\begin{proof}
Write $\qfibocat{m}{n} = \sum_{i=0}^D c_i q^i$ and $\qfibonom{m+n}{n} = \sum_{i=0}^d a_i q^i$, where $d = \deg \qfibonom{m+n}{n}$. Symmetry of $\qfibonom{m+n}{n}$ and $[F_{m+n}]_q$ together with \Cref{prop:fibocatintegral} implies symmetry of $\qfibocat{m}{n}$, so it suffices to show $c_i \geq 0$ for $0 \leq i \leq \lfloor \frac{D}{2} \rfloor$.
 
From
\[
\qfibocat{m}{n} = \frac{1-q}{1-q^{F_{m+n}}}\, \qfibonom{m+n}{n} = \sum_{k \geq 0}(1-q)\, q^{kF_{m+n}} \cdot \sum_{i=0}^d a_i q^i,
\]
we read off (with the convention $a_{-r} = 0$ for $r \geq 1$)
\[
c_i = \sum_{k \geq 0}\bigl(a_{i-kF_{m+n}} - a_{i-kF_{m+n}-1}\bigr).
\]
For $i \leq \lfloor \frac{D}{2} \rfloor \leq \lfloor \frac{d}{2} \rfloor$ and any $k \geq 0$, we have $i - kF_{m+n} \leq \lfloor \frac{d}{2} \rfloor$, so unimodality of $\qfibonom{m+n}{n}$ gives $a_{i-kF_{m+n}} - a_{i-kF_{m+n}-1} \geq 0$. Hence $c_i \geq 0$.
\end{proof}
 
We close with restating two open problems from \cite{kustner}.
 
\begin{problem}[{\cite[Open Problem 4.3]{kustner}}]
For $m, n \in \mathbb{Z}_{>0}$ with $\gcd(m,n) \in \{1,2\}$, find a combinatorial interpretation of $\qfibocat{m}{n}$.
\end{problem}

The Catalan case also remains open.

\begin{problem}[{\cite[Open Problem 5.4]{kustner}}]
For $m \geq 0$, find a combinatorial interpretation of $\qfibocat{m}{m+1}$.
\end{problem}

\bibliographystyle{alpha}
\bibliography{references}

@article{kustner,
  author = {Bergeron, Nantel and Ceballos, Cesar and Küstner, Josef},
  issn = {1815-0659},
  journal = {Symmetry, Integrability and Geometry: Methods and Applications},
  language = {eng},
  organization = {University of Vienna, Austria},
  title = {{Elliptic and $q$-Analogs of the Fibonomial Numbers}},
  year = {2020},
}

@book{benjamin,
  author = {Benjamin, Arthur and Quinn, Jennifer J.},
  address = {Washington, D.C.},
  isbn = {0883853337},
  language = {eng},
  lccn = {2003108524},
  publisher = {Mathematical Association of America},
  series = {Dolciani Mathematical Expositions; no.~27},
  title = {{Proofs That Really Count: The Art of Combinatorial Proof}},
  year = {2003},
}

@incollection{zeilberger1989one,
  title = {{A One-Line High School Algebra Proof of the Unimodality of the Gaussian Polynomials $\genfrac{[}{]}{0pt}{}{n}{k}$ for $k < 20$}},
  author = {Zeilberger, Doron},
  booktitle = {$q$-Series and Partitions},
  pages = {67--72},
  year = {1989},
  publisher = {Springer},
}

@article{sylvester1878xxv,
  title = {{XXV. Proof of the Hitherto Undemonstrated Fundamental Theorem of Invariants}},
  author = {Sylvester, James Joseph},
  journal = {The London, Edinburgh, and Dublin Philosophical Magazine and Journal of Science},
  volume = {5},
  number = {30},
  pages = {178--188},
  year = {1878},
  publisher = {Taylor \& Francis},
}

@article{o1990unimodality,
  title = {{Unimodality of Gaussian Coefficients: A Constructive Proof}},
  author = {O'Hara, Kathleen M.},
  journal = {Journal of Combinatorial Theory, Series A},
  volume = {53},
  number = {1},
  pages = {29--52},
  year = {1990},
  publisher = {Academic Press},
}

@article{sagan2009combinatorial,
  title = {{Combinatorial Interpretations of Binomial Coefficient Analogues Related to Lucas Sequences}},
  author = {Sagan, Bruce and Savage, Carla},
  journal = {arXiv preprint arXiv:0911.3159},
  year = {2009},
}

@article{zeckendorf1972representations,
  title = {{Représentation des nombres naturels par une somme de nombres de Fibonacci ou de nombres de Lucas}},
  author = {Zeckendorf, Édouard},
  journal = {Bulletin de la Société Royale des Sciences de Liège},
  pages = {179--182},
  year = {1972},
}

@article{adiprasito2018hodge,
  title = {{Hodge Theory for Combinatorial Geometries}},
  author = {Adiprasito, Karim and Huh, June and Katz, Eric},
  journal = {Annals of Mathematics},
  volume = {188},
  number = {2},
  pages = {381--452},
  year = {2018},
  publisher = {JSTOR},
}

@article{branden2020lorentzian,
  title = {{Lorentzian Polynomials}},
  author = {Brändén, Petter and Huh, June},
  journal = {Annals of Mathematics},
  volume = {192},
  number = {3},
  pages = {821--891},
  year = {2020},
  publisher = {Department of Mathematics, Princeton University},
}

@inproceedings{anari2018log,
  title = {{Log-Concave Polynomials, Entropy, and a Deterministic Approximation Algorithm for Counting Bases of Matroids}},
  author = {Anari, Nima and Gharan, Shayan Oveis and Vinzant, Cynthia},
  booktitle = {2018 IEEE 59th Annual Symposium on Foundations of Computer Science (FOCS)},
  pages = {35--46},
  year = {2018},
  organization = {IEEE},
}

@article{chan2024log,
  title = {{Log-Concave Poset Inequalities}},
  author = {Chan, Swee Hong and Pak, Igor},
  journal = {Journal of the Association for Mathematical Research},
  volume = {2},
  number = {1},
  pages = {53--153},
  year = {2024},
}

@article{stanley1989log,
  title = {{Log-Concave and Unimodal Sequences in Algebra, Combinatorics, and Geometry}},
  author = {Stanley, Richard P.},
  journal = {Annals of the New York Academy of Sciences},
  volume = {576},
  number = {1},
  pages = {500--535},
  year = {1989},
}

@article{bizley1954derivation,
  title = {{Derivation of a New Formula for the Number of Minimal Lattice Paths from $(0,0)$ to $(km,kn)$ Having Just $t$ Contacts with the Line $my=nx$ and Having No Points above This Line; and a Proof of Grossman's Formula for the Number of Paths Which May Touch but Do Not Rise above This Line}},
  author = {Bizley, Michael Terence Lewis},
  journal = {Journal of the Institute of Actuaries},
  volume = {80},
  number = {1},
  pages = {55--62},
  year = {1954},
  publisher = {Cambridge University Press},
}

@article{lucas1878theorie,
  title = {{Théorie des fonctions numériques simplement périodiques}},
  author = {Lucas, Édouard},
  journal = {American Journal of Mathematics},
  volume = {1},
  pages = {184--240},
  year = {1878},
  publisher = {JSTOR},
}

\end{document}